\documentclass[12pt,a4paper]{article}

\usepackage{amsthm,amsmath,natbib,amssymb,amsfonts,bm}
\usepackage{graphicx}
\usepackage{placeins}

\def\hang{\hangindent\parindent}
\def\rf{\par\noindent\hang}

\newtheorem{theorem}{Theorem}
\newtheorem{lemma}{Lemma}

\newtheorem*{corollary}{Corollary}
\newtheorem*{theorem*}{Theorem}

\theoremstyle{definition}

\theoremstyle{remark}

\DeclareMathAlphabet{\mathpzc}{OT1}{pzc}{m}{it}

\begin{document}

\baselineskip=20pt

\phantom{1}
\vspace{-2cm}
\begin{center}
 {\bf \Large On confidence intervals in regression that utilize uncertain prior information
 about a vector parameter}
\end{center}

\smallskip

\begin{center}
{\bf \large Paul Kabaila$^{\textstyle{^*}}$ and Dilshani Tissera}
\end{center}

\medskip

\noindent{\it Department of Mathematics and
Statistics, La Trobe University, Victoria 3086, \newline Australia}

\medskip
\noindent{\bf Abstract}
\medskip

Consider a linear regression model with $n$-dimensional response vector,
$p$-dimensional regression parameter $\boldsymbol{\beta}$ and
independent normally distributed errors.
Suppose that the parameter of interest is $\theta = \boldsymbol{a}^T \boldsymbol{\beta}$
where $\boldsymbol{a}$ is a specified vector.
Define the $s$-dimensional parameter vector $\boldsymbol{\tau}=\boldsymbol{C}^T \boldsymbol{\beta} - \boldsymbol{t}$
where $\boldsymbol{C}$ and $\boldsymbol{t}$ are specified.
Also suppose that we have uncertain prior information that $\boldsymbol{\tau}=\boldsymbol{0}$.
Part of our evaluation of a frequentist confidence interval for $\theta$ is the ratio
(expected length of this confidence interval)/(expected length of standard $1-\alpha$ confidence interval),
which we call the scaled expected length of this interval.
We say that a $1-\alpha$ confidence interval for $\theta$ utilizes this
uncertain prior information if (a) the scaled expected length of this interval is significantly
less than 1 when $\boldsymbol{\tau}=\boldsymbol{0}$, (b) the maximum value
of the scaled expected length is
not too large  and (c) this
confidence interval reverts to the standard $1-\alpha$ confidence interval
when the data happen to strongly contradict the prior information.
Let $\hat \Theta=\boldsymbol{a}^T \hat{\boldsymbol{\beta}}$ and
$\hat{\boldsymbol{\tau}} = \boldsymbol{C}^T \hat{\boldsymbol{\beta}} - \boldsymbol{t}$,
where
$\hat{\boldsymbol{\beta}}$ is the least squares estimator of $\boldsymbol{\beta}$.
We consider the particular case that that $E \big ( (\hat{\boldsymbol{\tau}} - \boldsymbol{\tau}) (\hat \Theta - \theta) \big ) = \boldsymbol{0}$,
so that $\hat \Theta$ and $\hat{\boldsymbol{\tau}}$ are independent.
We present a new $1-\alpha$ confidence interval for $\theta$ that utilizes the uncertain prior information
that $\boldsymbol{\tau}=\boldsymbol{0}$.
The following problem is used to
illustrate the application of this new confidence interval.
Consider a $2^3$ factorial experiment with 1 replicate.
Suppose that the parameter of interest $\theta$ is a specified linear combination of the main effects. Assume that the three-factor
interaction is zero. Also suppose that we have
uncertain prior information that
all of the two-factor interactions are zero. Our aim is to find a frequentist
0.95 confidence interval for $\theta$ that utilizes this uncertain
prior information.

\bigskip

\noindent {\sl Keywords:} Frequentist confidence interval; Prior information;
Linear regression.


\bigskip

\noindent $^*$ Corresponding author. Address: Department of
Mathematics and Statistics, La Trobe University, Victoria 3086,
Australia; Tel.: +61-3-9479-2594; fax: +61-3-9479-2466.
{\it E-mail address:} P.Kabaila@latrobe.edu.au.

\newpage

\baselineskip=20pt

\noindent {\bf 1. Introduction}

\medskip

Suppose that the parameter of interest $\theta$ is a scalar and that we have
uncertain prior information about the parameters of the model.
Hodges and Lehmann (1952), Bickel (1984) and Kempthorne (1983, 1987, 1988) show how such uncertain prior information
can be utilized in frequentist inference, mostly for point estimation of $\theta$.
A confidence interval for $\theta$ is said to be a $1-\alpha$ confidence
interval if it has infimum coverage probability $1-\alpha$.
We assess a $1-\alpha$ confidence
interval $J$ by its scaled expected length, defined to be the ratio (expected length of $J$)/(expected
length of the standard $1-\alpha$ confidence interval for $\theta$).
The first requirement of a
$1-\alpha$ confidence interval that utilizes the uncertain
prior information is that its scaled expected length is significantly less than 1 when the
prior information is correct (Kabaila, 2009).

We classify confidence intervals that satisfy this first requirement
into the following two groups. The first group consists of
$1-\alpha$ confidence intervals with scaled expected length
that is less than or equal to 1 for all parameter values, so that these dominate the
standard $1-\alpha$ confidence interval. An example of such a confidence interval is the Stein-type confidence interval
for the normal variance (see e.g. Maata and Casella, 1990 and Goutis and Casella, 1991). The second group consists of $1-\alpha$
confidence intervals that satisfy this first requirement, when dominance of the usual $1-\alpha$ confidence interval is not possible (the
scaled expected length must exceed 1 for some parameter values).
Some relevant admissibility results are provided by
Kabaila, Giri and Leeb (2010) and Kabaila (2011).
This second
group includes the confidence intervals described by Pratt (1961), Brown et al (1995)
and Puza and O'Neill (2006ab). This second group also includes
$1-\alpha$ confidence intervals that satisfy the additional requirements that (a) the maximum (over the parameter space)
of the scaled expected length is not too much larger than 1 and (b) the confidence interval reverts to
the standard $1-\alpha$ confidence interval when the data happen to strongly contradict the prior information.
Confidence intervals that utilize uncertain the prior information and satisfy these additional requirements have been proposed by Farchione and Kabaila (2008)
and Kabaila and Giri (2009ab) (cf Kabaila and Giri, 2013).

Consider the linear regression model $\boldsymbol{Y}=\boldsymbol{X} \boldsymbol{\beta} + \boldsymbol{\varepsilon}$ ,
where $\boldsymbol{Y}$ is a random $n$-vector of responses, $\boldsymbol{X}$ is a $n \times p$ matrix with linearly independent columns,
$\boldsymbol{\beta}$ is an unknown parameter $p$-vector and  $\varepsilon \sim N(\boldsymbol{0},
\sigma^2 \boldsymbol{I}_n)$, where $\sigma^2$ is an unknown positive parameter.
Suppose that the parameter of interest is $\theta = \boldsymbol{a}^T \boldsymbol{\beta}$,
where $\boldsymbol{a}$ is a given $p$-vector $(\boldsymbol{a} \neq \boldsymbol{0})$.
Let the $s$-dimensional parameter vector $\boldsymbol{\tau}$ be defined to be $\boldsymbol{C}^T \boldsymbol{\beta} - \boldsymbol{t}$
where $\boldsymbol{C}$ is a specified $p \times s$ matrix  $(s < p)$ with linearly independent columns and $\boldsymbol{t}$ is a specified $s$-vector. Suppose that $\boldsymbol{a}$ does not belong to the linear subspace
spanned by the columns of $\boldsymbol{C}$.
Also suppose that previous experience with similar data sets and/or expert opinion and scientific background suggests
that $\boldsymbol{\tau}=\boldsymbol{0}$. In other words, suppose that we have uncertain prior information that $\boldsymbol{\tau}=\boldsymbol{0}$.
Our aim is to find a frequentist $1-\alpha$ confidence interval
for $\theta$ that utilizes this prior information.
By ``utilizes this prior information'' we mean that (a) the scaled expected length of this interval is significantly
less than 1 when $\boldsymbol{\tau}=\boldsymbol{0}$, (b) the maximum value
of the scaled expected length is
not too large and (c) this
confidence interval reverts to the standard $1-\alpha$ confidence interval
when the data happen to strongly contradict the prior information.

Kabaila and Giri (2009a) have dealt with the case that $s=1$, so we consider the case that $s \ge 2$.
Let $\hat{\boldsymbol{\beta}}$ denote the least squares estimator of $\boldsymbol{\beta}$.
Also let
$\hat \Sigma^2 = (\boldsymbol{Y} - \boldsymbol{X} \hat{\boldsymbol{\beta}})^T (\boldsymbol{Y} - \boldsymbol{X} \hat{\boldsymbol{\beta}})/(n-p)$,
$\hat \Theta = \boldsymbol{a}^T \hat{\boldsymbol{\beta}}$ and
$\hat{\boldsymbol{\tau}} = \boldsymbol{C}^T \hat{\boldsymbol{\beta}} - \boldsymbol{t}$.
 We consider the particular case
that $E \big ( (\hat{\boldsymbol{\tau}} - \boldsymbol{\tau}) (\hat \Theta - \theta) \big ) = \boldsymbol{0}$.
An example of this particular case is the following. Consider a $2^3$ factorial experiment with 1 replicate.
For factorial experiments, it is a widely-held belief that the higher the order of interaction, the more likely it is to be negligible.
Indeed, fractional factorial designs are based on this belief. Assume that the third-order interaction is zero.
Also suppose that we have uncertain prior information that all of the second-order interactions are zero. In this case,
$n-p=1$ and $s=3$. If the parameter of interest $\theta$ is a linear combination of the main effects then
$E \big ( (\hat{\boldsymbol{\tau}} - \boldsymbol{\tau}) (\hat \Theta - \theta) \big ) = \boldsymbol{0}$.

In Section 2, we describe the new $1-\alpha$ confidence interval for $\theta$ that utilizes the uncertain prior information
that $\boldsymbol{\tau}=\boldsymbol{0}$.
Define $\boldsymbol{\gamma} = \big ( \text{Cov}(\hat{\boldsymbol{\tau}}) \big)^{-1/2} \, \boldsymbol{\tau}$.
The coverage probability and scaled expected length of this new confidence interval are even functions of $||\boldsymbol{\gamma}||
= \sqrt{\boldsymbol{\gamma}^T \boldsymbol{\gamma}}$.
In Section 3, we consider this $2^3$ factorial experiment,
when $1-\alpha=0.95$. Figure 2 presents graphs of the squared scaled expected length and the coverage probability
of this new confidence interval (as functions of $||\boldsymbol{\gamma}||$).
The infimum coverage probability is computed to be 0.95.
To an excellent approximation, the coverage probability of the new confidence interval is equal to 0.95, throughout the
parameter space. This figure demonstrates that this new confidence interval has excellent performance in terms of
squared scaled expected length.
When the prior information is correct (i.e. $||\boldsymbol{\gamma}||=0$), we gain since the square of the scaled expected length
is 0.34707, which is much smaller than 1. The maximum value of the square of the scaled expected length is
1.0404, which is
only slightly larger than 1.
The new 0.95 confidence interval for $\theta$ coincides with the standard $1-\alpha$ confidence interval
when the data strongly contradicts the prior information. This is reflected in Figure 2 by the fact that the
square of the scaled expected length approaches 1 as $||\boldsymbol{\gamma}|| \rightarrow \infty$.
In Section 4, we examine the effect on the performance of the new confidence interval of increasing $s$,
for $n-p=1$. The application of the new confidence interval is to the case that $n-p$ is small. As pointed
out in Section 2, the ability of the new confidence interval to utilize the uncertain prior information comes
from enhanced estimation of $\sigma^2$. The smaller $n-p$ is, the larger will be this enhancement.

\bigskip

\noindent {\bf 2. New confidence interval that utilizes the uncertain prior information}

\medskip

Our first step is to reduce the data to $\big( \hat \Theta, \hat{\boldsymbol{\tau}}, \hat \Sigma^2 \big)$.
Let $v_{11} = \boldsymbol{a}^T (\boldsymbol{X}^T \boldsymbol{X})^{-1} \boldsymbol{a}$ and
$\boldsymbol{V}_{22} = \boldsymbol{C}^T (\boldsymbol{X}^T \boldsymbol{X})^{-1} \boldsymbol{C}$.
Note that $\hat \Theta \sim N(\theta, \sigma^2 v_{11})$, $\hat{\boldsymbol{\tau}} \sim N(\boldsymbol{\tau},
\sigma^2 \boldsymbol{V}_{22})$ and  $\hat \Sigma^2$ are independent random vectors. Let $m=n-p$.
Define the quantile $t(m)$ by the requirement that
$P \big(-t(m) \le T \le t(m) \big) = 1-\alpha$ for $T \sim t_m$.
The standard $1-\alpha$ confidence interval for $\theta =\boldsymbol{a}^T \boldsymbol{\beta}$ is
\begin{equation*}
I = \big[ \hat \Theta  -  t(m)  \sqrt{v_{11}} \  \hat \Sigma ,  \, \hat \Theta  +  t(m)  \sqrt{v_{11}} \  \hat \Sigma \big].
\end{equation*}

The new confidence interval for $\theta$ that we will describe shortly  is centered at $\hat \Theta$.
The fact that $\hat \Theta$ and
$\hat{\boldsymbol{\tau}}$ are independent suggests that the uncertain prior information that $\boldsymbol{\tau}=\boldsymbol{0}$
should not influence the point estimation of $\theta$.
However, this uncertain prior information can be used to enhance the estimation of $\sigma^2$.
In the absence of any prior information about
$\boldsymbol{\tau}$, the standard estimator of $\sigma^2$ is $\hat \Sigma^2$
and $m \hat \Sigma^2 / \sigma^2 \sim \chi_m^2$. However, if it known that
$\boldsymbol{\tau}=\boldsymbol{0}$ then that standard estimator of $\sigma^2$ is
\begin{equation*}
\tilde{\Sigma}^2 = \frac{m \, \hat \Sigma^2 + \hat{\boldsymbol{\tau}}^T \, \boldsymbol{V}_{22}^{-1} \, \hat{\boldsymbol{\tau}}}
{m+s} = \hat \Sigma^2 \left ( \frac{m + s F}{m + s} \right ),
\end{equation*}
where $F= \big(\hat{\boldsymbol{\tau}}^T \, \boldsymbol{V}_{22}^{-1} \, \hat{\boldsymbol{\tau}}/s \big)/\hat{\Sigma}^2$
and $(m+s) \tilde \Sigma^2 / \sigma^2 \sim \chi_{m+s}^2$.
This suggests that the uncertain prior information that $\boldsymbol{\tau}=\boldsymbol{0}$ can be used to enhance
the estimation of $\sigma^2$
by using the appropriate function of $\hat \Sigma^2$ and $F$.

This motivates us to consider a new confidence interval for $\theta$ of the form
\begin{equation*}
J(d) = \Big [ \hat \Theta - \sqrt{v_{11}} \, \hat \Sigma \, d ( \sqrt{F} ), \,
\hat \Theta + \sqrt{v_{11}} \, \hat \Sigma \, d ( \sqrt{F} ) \Big ],
\end{equation*}
where the function $d: [0,\infty) \rightarrow (0,\infty)$ is required to satisfy the following restrictions.

\medskip

\noindent \textbf{Restriction 1} \ $d$ is a continuous function.

\medskip

\noindent \textbf{Restriction 2} \ $d(x) = t(m)$ for all $x \ge k$, where $k$ is a specified positive number.

\medskip

\noindent The first restriction implies that the endpoints of the confidence interval $J(d)$ are continuous
functions of the data. Note that $F$ is the usual F statistic for testing the null
hypothesis $H_0: \boldsymbol{\tau} = \boldsymbol{0}$ against the alternative hypothesis
$H_1: \boldsymbol{\tau} \ne \boldsymbol{0}$.
Thus the second restriction implies that this confidence interval reverts to the usual
$1-\alpha$ confidence interval $I$ when the data happen to strongly contradict the prior information that
$\boldsymbol{\tau}=\boldsymbol{0}$.

Part of the evaluation of the confidence interval $J(d)$ consists of comparing it with the
usual $1-\alpha$ confidence interval $I$ using the scaled expected length criterion
(expected length of this confidence interval) / (expected length of $I$).
Theorem 1, which is stated and proved in Appendix A, provides computationally-convenient expressions for the coverage probability and
scaled expected length of $J(d)$.
Define
$\boldsymbol{\gamma} = \big ( \text{Cov}(\hat{\boldsymbol{\tau}}) \big)^{-1/2} \, \boldsymbol{\tau} = (1/\sigma) \boldsymbol{V}_{22}^{-1/2} \boldsymbol{\tau}$.
According to this theorem, for given function $d$, both the
coverage probability and the scaled expected length of $J(d)$ are functions of
$||\boldsymbol{\gamma}||$. We denote this scaled expected length by
$e \big(||\boldsymbol{\gamma}||; d \big)$.
The numerical integration method used to evaluate this coverage probability is described in Appendix B.

Our aim is to find a function $d$ satisfying Restrictions 1 and 2 and such that (a) the minimum of
$P(\theta \in J(d))$ over $||\boldsymbol{\gamma}||$ is $1-\alpha$ and (b)
$e(0; d)$ is minimized subject to the restriction that $e \big(||\boldsymbol{\gamma}||; d \big) \le \ell$
for all $||\boldsymbol{\gamma}||$, where $\ell \ge 1$ is chosen
by the statistician prior to the analysis of the data. Theorem 2, which is stated and proved in Appendix C,
provides a computationally-convenient
expression for $e(0; d)$. We expect that, for small $m$, this constrained minimization
will lead to a $1-\alpha$ confidence interval for $\theta$ that has scaled expected length
$e \big(||\boldsymbol{\gamma}||; d \big)$ that
is substantially less 1 for $\boldsymbol{\gamma}=\boldsymbol{0}$.

We implement the coverage constraint $P(\theta \in J(d)) \ge 1-\alpha$ for all $||\boldsymbol{\gamma}|| \ge 0$
as follows. For any reasonable choice of the function $d$, $P(\theta \in J(d))$ converges to $1-\alpha$ as
$||\boldsymbol{\gamma}|| \rightarrow \infty$. The constraints implemented in the computations are that
$P(\theta \in J(d)) \ge 1-\alpha$ for every $||\boldsymbol{\gamma}||$ in a judiciously-chosen finite set of
values ${\cal G}$. That a given ${\cal G}$ is adequate to the task is judged by checking numerically, at the
completion of the computations, that the coverage probability constraint is satisfied for all $||\boldsymbol{\gamma}|| \ge 0$
(cf. Farchione and Kabaila, 2012).

For computational feasibility, we specify the following parametric form for the
function $d$.
Suppose that $x_1, \ldots, x_q$
satisfy $0 = x_1 < x_2 < \cdots < x_q = k$.
We fully specify the function $d$ by the vector $\big (d(x_1), \ldots, d(x_{q-1}) \big)$
as follows.
The value of $d(x)$ for any $x \in [0,k]$ is specified
by natural cubic spline interpolation for these given function values
and $d(x_q)=t(m)$ (without any endpoint conditions
on the first derivative of $d$). We call $x_1, x_2, \ldots x_q$ the knots.
Of course, the values of $k$, $\ell$ and knots $x_i$ need to be judiciously-chosen and
this will usually require some computational exploration.

\bigskip

\noindent {\bf 3. Application to the analysis of data from a
single-replicate $\boldsymbol{2^3}$ factorial experiment}

\medskip

Consider a $2^3$ factorial experiment carried out without replication. Let $Y$ denote the response and let
$x_1$, $x_2$ and $x_3$ denote the coded levels for each of the 3 factors, where the coded level
takes either the value $-1$ or 1. We assume the model
\begin{equation*}
Y=\beta_{0}+\beta_{1}x_{1} +\beta_{2}x_{2}+\beta_{3}x_{3}+\beta_{12}x_{1}x_{2}
+\beta_{13}x_{1}x_{3}+\beta_{23}x_{2}x_{3}+\beta_{123}x_{1}x_{2}x_{3}+ \varepsilon
\end{equation*}
where  $\beta_{0}$,
$\beta_{1}$, $\beta_{2}$, $\beta_{3}$, $\beta_{12}$, $\beta_{13}$,
$\beta_{23}$, $\beta_{123}$ are unknown parameters and
$\varepsilon \sim N(0, \sigma^2 )$, where $\sigma^2$ is an unknown
positive parameter.

For factorial experiments it is
commonly believed that higher order interactions are negligible
(see e.g. Mead (1988, p.368) and Hinkelman \& Kempthorne (1994, p.350)). Indeed, this type of
belief is the basis for the design of fractional factorial experiments.
Assume that $\beta_{123} = 0$. Also suppose that we have uncertain prior information that
$\beta_{12}$, $\beta_{13}$ and $\beta_{23}$ are all zero. Thus $n-p = 1$.
We consider the particular case that the parameter of interest
interest $\theta$ is a linear combination of the main effects i.e.
$\theta = a_1 \beta_1 + a_2 \beta_2 + a_3 \beta_3$. In this case,
$E \big ( (\hat{\boldsymbol{\tau}} - \boldsymbol{\tau}) (\hat \Theta - \theta) \big ) = \boldsymbol{0}$.

Of course, the properties of $J(d)$, resulting from the constrained minimization described in Section 2, depend on the
values of $k$, $\ell$, the knots $x_i$ and $1-\alpha$. We focus on the particular case that $k=15$, $\ell=1.02$, the knots
are at $0, 1, 2, 3, 7, 12, 15$ and $1-\alpha = 0.95$. When we compute the new confidence interval, we obtain the
function $d$ shown in Figure 1.
All of the computations presented in the present
paper were performed with programs written in MATLAB, using the optimization and
statistics toolboxes.
Consistent with the corollary stated in Appendix C,
$d(x)$ takes values larger than $t(m)$.
Figure 2 presents graphs of the squared scaled expected length and the coverage probability
(as functions of $||\boldsymbol{\gamma}||$) of this new confidence interval.
The squared scaled expected length and coverage probability computations were checked
using Monte Carlo simulations.

The infimum coverage probability is computed to be 0.95.
The upper panel of Figure 2 demonstrates that this new confidence interval has excellent performance in terms of
squared scaled expected length.
When the prior information is correct (i.e. $||\boldsymbol{\gamma}||=0$), we gain since the square of the scaled expected length
is 0.34707, which is much smaller than 1. The maximum value of the square of the scaled expected length is
1.0404, which is
only slightly larger than 1.
The new 0.95 confidence interval for $\theta$ coincides with the standard $1-\alpha$ confidence interval
when the data strongly contradicts the prior information. This is reflected in the upper panel of Figure 2 by the fact that the
square of the scaled expected length approaches 1 as $||\boldsymbol{\gamma}|| \rightarrow \infty$.

\begin{figure}[p]
\includegraphics{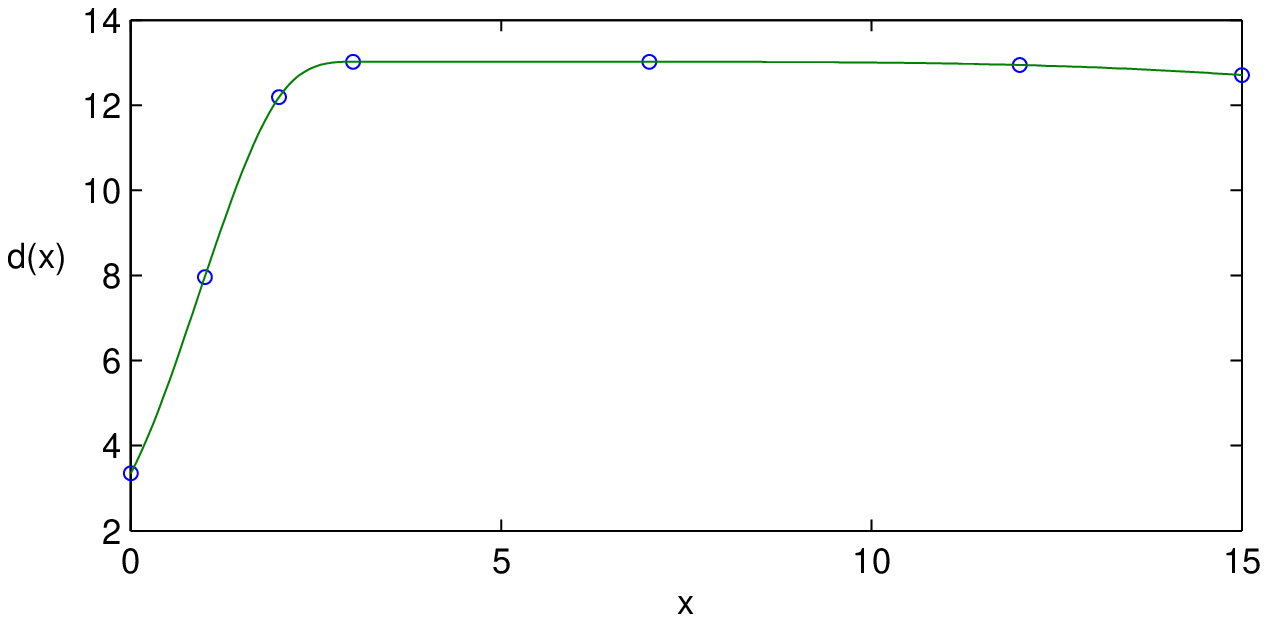}
\caption{\small{Plot of the function $d$ for the new 0.95 confidence interval for $\theta$ when $s=3$, $m=n-p=1$
and $\ell=1.02$. The knots are at 0, 1, 2, 3, 7, 12, 15.}}
\label{fig:spline}
\includegraphics{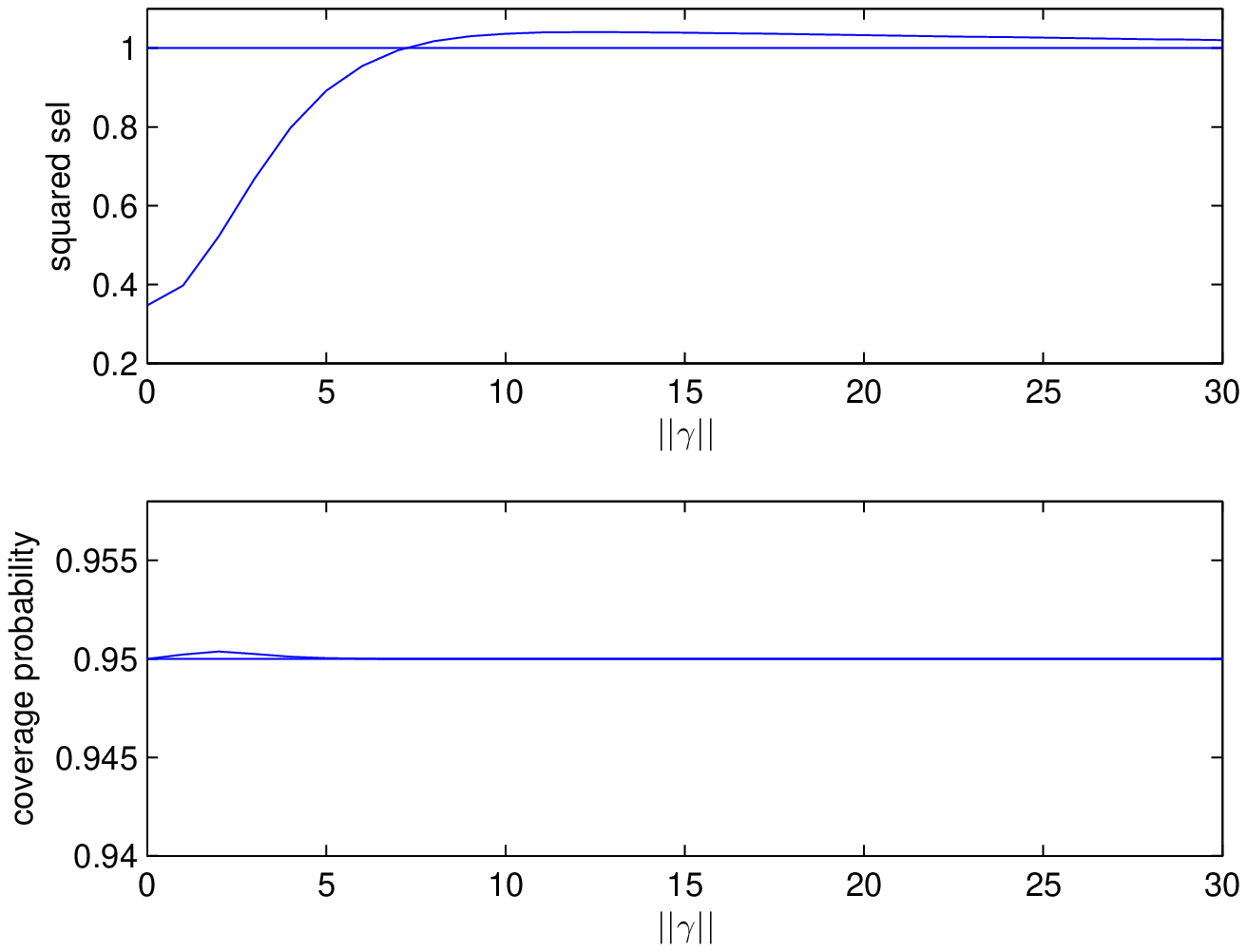}
\caption{\small{Plots of squared scaled expected length $e^2(||\gamma||;d)$ and coverage probability as function of $||\gamma||$ when $1-\alpha=0.95$, $s=3$, $m=n-p=1$ and $\ell=1.02$. }}
\label{fig:bothSelCov}
\end{figure}

\bigskip

\noindent {\bf 4. The effect on the performance of the new confidence interval of increasing the
value of $\boldsymbol{s}$}

\medskip

Suppose that the value of $m$ is fixed and that we increase $s$.
It seems plausible that the best possible
performance of the new confidence interval for $\theta$ will increase as $s$
increases i.e. as the amount of uncertain prior information increases.
We have examined the truth of this plausible result as follows. We have considered
$m=1$ and $1-\alpha=0.95$, chosen $\ell = 1.02$ and the number of knots to be 7.
For each $s=1, 2, 3, 5$ and 7, we have chosen $k$ and the knots so as to minimize
the scaled expected length at $||\boldsymbol{\gamma}||=0$.
We have obtained the following results:
\begin{table}[h]
\begin{center}
\begin{tabular}{|c|r|r|r|r|}\hline
s&Min sq sel&Max sq sel&Min CP&Max CP  \\\hline
    1&0.80549&1.0414&0.95&0.95049 \\\hline
    2&0.54698&1.0404&0.95&0.95030 \\\hline
    3&0.34707&1.0404&0.95&0.95037 \\\hline
    5&0.25151&1.0406&0.95&0.95034 \\\hline
    7&0.19027&1.0404&0.95&0.95030 \\\hline
\end{tabular}
\end{center}
\caption{Comparison of results for different values of $s$, for $m=1$, $1-\alpha=0.95$ and
$\ell =1.02$. The column labels Min sq sel,
Max sq sel, Min CP and Max CP
denote the minimum squared scaled expected length, the
maximum squared scaled expected length, the minimum coverage
probability and the maximum coverage probability, respectively.}
\label{comparison_diff_s}
\end{table}

\noindent If, for each $s$ considered, we assume that the performance of the confidence
interval in terms of the scaled expected length at $||\boldsymbol{\gamma}||=0$
is about as good as it can be then this table tells us the following. As $s$
increases, the amount of uncertain prior information increases and this leads to
an improvement in the performance of this confidence interval.

\bigskip

\noindent {\bf 5. Discussion}

\medskip

In this paper we have shown how to construct a frequentist confidence interval for the parameter of interest
$\theta$ that utilizes the uncertain prior information that $\boldsymbol{\tau} = \boldsymbol{0}$.
We have done this for the particular case that the covariances
between the components of the least squares estimator of $\boldsymbol{\tau}$ and the least squares estimator
of $\theta$ are all zero.
Our practical experience with the computations of this new confidence interval, in a variety of circumstances,
shows that the coverage probability needs to be computed with great accuracy for these computations to be
successful. If we no longer restrict attention to the particular case that all of these covariances are zero
then the construction of such a confidence interval necessitates
the computation of coverage probabilities using
more complicated methods of the type employed by Kabaila and Farchione (2012).
However, the increased computation time of these methods would appear
to make the constrained optimization not computationally practicable.

\bigskip

\noindent \textbf{Appendix A. Theorem 1 and its proof}

\medskip

In this appendix, we state and prove Theorem 1, which provides new computationally-convenient expressions
for the coverage probability and scaled expected length of the confidence interval $J(d)$. Define
$G=(\hat \Theta - \theta)/(\sigma \sqrt{v_{11}})$ and
$Q = (1/\sigma^2) \hat{\boldsymbol{\tau}}^T \boldsymbol{V}_{22}^{-1} \hat{\boldsymbol{\tau}}$.
Also define $V=\sqrt{Q/s}$ and
$W=\hat{\Sigma} /\sigma$. Now, $(G,V)$ and $W$ are independent random vectors.
The assumption that
$E \big ( (\hat{\boldsymbol{\tau}} - \boldsymbol{\tau}) (\hat \Theta - \theta) \big ) = \boldsymbol{0}$
implies that $G$ and $V$ are independent random variables. Thus, $G$, $V$ and $W$ are independent random
variables. Note that $G \sim N(0,1)$, $Q$ has a noncentral $\chi^2$ distribution with $s$ degrees of
freedom and noncentrality parameter $||\boldsymbol{\gamma}||^2= \boldsymbol{\gamma}^T \boldsymbol{\gamma}$ and $W$ has the same distribution as
$\sqrt{\chi_m^2/m}$.

\medskip

\begin{theorem}

Let $f_V( \,\cdot \,; ||\boldsymbol{\gamma}||)$ and $f_W$ denote the probability density functions of $V$ and $W$, respectively.
Also let $\Phi$ denote the $N(0,1)$ distribution function.

\begin{enumerate}

\item[(a)]

The coverage probability of $J(d)$
is equal
to
\begin{equation}
\label{cov_pr_final}
1 - \alpha + 2 \int_0^k \int_0^{\infty}
\big ( \Phi (w \, d (x) ) - \Phi (w \, t(m) )\big )
\, f_V(xw; ||\boldsymbol{\gamma}||) \,
w \, f_W(w) \, dw \, dx.
\end{equation}
For given function $d$, the
coverage probability of $J(d)$ is a function of
$||\boldsymbol{\gamma}||$.

\item[(b)]

The scaled expected length of $J(d)$
is equal
to
\begin{equation}
\label{sc_exp_len_final}
1 +
\frac{1}{t(m) E(W)} \int_0^{\infty} \int_0^k \big( d(x) - t(m) \big) \, f_V(x w; ||\boldsymbol{\gamma}||)
\, dx \, w^2 \, f_W(w) \, dw.
\end{equation}
For given function $d$, the
scaled expected length of $J(d)$ is a function of
$||\boldsymbol{\gamma}||$.

\end{enumerate}

\end{theorem}

\noindent \textbf{Proof of part (a).} \ It is straightforward to show that the coverage probability $P(\theta \in J(d))$ is equal to
$P \big( - W \, d (V/W) \le G \le W \, d (V/W) \big)$.
By the law of total probability, this is equal to
\begin{equation*}
P \big( - W \, d (V/W) \le G \le W \, d (V/W), V/W<k \big)  + P \big( - t(m) W \,  \le G \le t(m) W , \, V/W\ge k \big),
\end{equation*}
since $d(x)=t(m)$ for all $x \ge k$ . Now
\begin{align*}
&P \big( - t(m) W \,  \le G \le t(m) W , \, V/W\ge k \big) + P \big( - t(m) W \,  \le G \le t(m) W , \, V/W < k \big) \nonumber  \\
&=P \big( - t(m) W \,  \le G \le t(m) W  \, \big)   \\
&=P \big( - t(m) \,  \le G/W \le t(m)  \, \big) = 1- \alpha
\end{align*}
Thus $P(\theta \in J(d))$ is equal to
\begin{align*}
&1 - \alpha
+ P \big( - W \, d (V/W) \le G \le W \, d (V/W), V/W<k \big) \\
&\ \ \ \ \ \ \
- P \big( - t(m) W \le G \le t(m) W, V/W<k \big) \\
&=1 - \alpha + \int_0^{\infty} \int_0^{kw} \big ( 2 \Phi \left ( w d(v/w) \right )\,  - 1 \, \big ) f_V(v) dv \, \, f_W(w) dw \\
&\ \ \ \ \ \ \ \ \ \ \  - \int_0^{\infty} \int_0^{kw} \big ( 2 \Phi \left ( w  t(m) \right )\,  - 1 \, \big ) f_V(v) dv \, f_W(w) dw \nonumber \\
&=1 - \alpha +  2 \int_0^{\infty} \int_0^{kw} \big ( \Phi \left ( w d(v/w) \right )\,  - \Phi \left( w t(m)  \right ) \, \big ) f_V(v) dv \, \, f_W(w) dw.
\end{align*}
Changing the variable of integration of the inner integral to $x=v/w$, we obtain
\begin{align*}
&1 - \alpha +  2 \int_0^{\infty} \int_0^{k} \big ( \Phi \left ( w d(x) \right )\,  - \Phi \left( w t(m)  \right ) \, \big ) f_V(xw;||\gamma||)\, dx \, \, w f_W(w) dw \nonumber \\
&=1 - \alpha +  2 \int_0^{k} \int_0^{\infty} \big ( \Phi \left ( w d(x) \right )\,  - \Phi \left( w t(m)  \right ) \, \big ) f_V(xw;||\gamma||)  \, w f_W(w) dw \, \, dx
\end{align*}

\bigskip

\noindent \textbf{Proof of part (b).} \  It is straightforward to show that the scaled expected length of $J(d)$ is
equal to
\begin{equation}
\label{sc_exp_len_first}
\frac{E \big ( W \, d(V/W) \big )}{t(m) \, E(W)}.
\end{equation}
We use the notation
\begin{equation*}
{\cal I}({\cal A}) =
\begin{cases}
1 &\text{if } {\cal A} \ \ \text{is true} \\
0 &\text{if } {\cal A} \ \ \text{is false}
\end{cases}
\end{equation*}
where ${\cal A}$ is an arbitrary statement.
Since ${\cal I}(V/W < k)+{\cal I}(V/W \ge k) = 1$, $E \big ( W \, d(V/W) \big )$ is equal to
\begin{align*}
&E \big( \, W d(V/W)\, {\cal I}(V/W < k) \, \big ) +  E \big( \, W d(Q/W^2) \, {\cal I}(V/W \ge k) \, \big ) \\
&=  E \big( \, W d(V/W) \, {\cal I}(V/W < k) \, \big ) +  E \big( \, W t(m) \, {\cal I}(V/W \ge k) \, \big ) \\
&=  t(m) E(W) + E \big( \, (d(V/W) - t(m)) \, W \, {\cal I}(V/W < k) \, \big ).
\end{align*}
Thus the expression \eqref{sc_exp_len_first} for the scaled expected length is equal to
\begin{align}
\label{penult_sc_exp_len_1}
&1 + \frac{1}{t(m) E(W)} E \big( \, (d(V/W) - t(m)) \, W \, {\cal I}(V/W < k) \, \big ) \\
\label{penult_sc_exp_len_2}
&= 1 + \frac{1}{t(m) E(W)} \int_0^{\infty} \int_0^{k w} \left ( d \left ( \frac{v}{w} \right) - t(m) \right )
f_V(v; ||\gamma||) \, dv \, w \, f_W(w) \, dw.
\end{align}
Changing the variable of integration of the inner integral to $x = v/w$, we obtain
\eqref{sc_exp_len_final}.

\bigskip

\noindent \textbf{Appendix B. The method used to evaluate the coverage probability}

\medskip

In this appendix we describe the numerical integration method used to compute the coverage probability
$P(\theta \in J(d))$, as given by \eqref{cov_pr_final}.
The function $d$, which is a cubic spline in the interval
$[0,k]$, does not necessarily have a third derivative
at each of the knots $x_2, \dots, x_q=k$. So we evaluate
\eqref{cov_pr_final} by computing
\begin{equation}
\label{second_formula}
1-\alpha + 2 \sum_{i=1}^{q-1} \int_{x_i}^{x_{i+1}} \int_0^{\infty} \big( \Phi(w d(w)) - \Phi(w t(m)) \big) f_V(xw; ||\gamma||) \, w \, f_W(w) \, dw \, dx.
\end{equation}
Each of the inner integrals is equal to
\begin{align}
\label{inner_int}
&\int_0^{\infty} \big( \Phi(w d(w)) - \Phi(w t(m)) \big) f_V(xw; ||\gamma||) \, w \, f_W(w) \, dw
\notag \\
&= E \Big ( \big ( \Phi(W d(x))\,  - \, \Phi(W t(m))\big ) \, f_V(xW;||\gamma ||)\, W \Big ),
\end{align}
where $W$ has pdf $f_W$.
Let $Z=m W^2$, so that $Z \sim \chi^2_{m}$.
Thus \eqref{inner_int} is equal to
\begin{align}
\label{inner_int_final}
&E \left \{ \left ( \Phi \left(\sqrt{\frac{Z}{m}} \,  d(x) \right)\,  - \, \Phi \left(\sqrt{\frac{Z}{m}} \, t(m) \right)\right ) \, f_V \left(x\sqrt{\frac{Z}{m}} ;||\gamma || \right)\, \sqrt{\frac{Z}{m}}  \right \}
\notag \\
&=\frac{1}{\sqrt{m}} \int_0^{\infty} \bigg ( \Phi \left ( \sqrt{\frac{z}{m}} \, d(x) \right )\,  - \Phi \left( \sqrt{\frac{z}{m}} \, t(m)  \right )  \bigg ) f_V \left(x \sqrt{\frac{z}{m}};||\gamma || \right)\sqrt{z}\, f_m(z) \, dz
\end{align}
where $f_m$ denotes the $\chi^2_{m}$ pdf.
It can be shown that
\begin{equation*}
z^{1/2} f_m(z) = \frac{2^{1/2} \, \Gamma ((m+1)/2)}{\Gamma(m/2)} f_{m+1}(z).
\end{equation*}
Thus \eqref{inner_int_final} is equal to
\begin{equation}
\label{inner_int_tmp}
\sqrt{\frac{2}{m}} \frac {\Gamma ((m+1)/2)}{ \Gamma(m/2) } \int_0^{\infty} \Big ( \Phi \left ( \sqrt{\frac{z}{m}} d(x) \right )\,  - \Phi \left( \sqrt{\frac{z}{m}} t(m)  \right )  \Big ) f_V \left ( x \sqrt{\frac{z}{m}};||\gamma || \right ) f_{m+1}(z) dz.
\end{equation}
Assuming that $d(x) > 0$ for all $x \in [0,k]$, we compute this as follows.
Let $F_m$ denote the $\chi^2_{m}$ cdf. Now change the variable of integration to
$u=F_{m+1}(z)$, so that \eqref{inner_int_tmp} is equal to
\begin{equation*}
\sqrt{\frac{2}{m}} \, \frac{ \Gamma ((m+1)/2)}{ \Gamma(m/2)}  \int_0^{1} g \big(u;x,||\gamma|| \big) \, du
\end{equation*}
where $g(u;x,||\gamma||)$ is defined to be
\begin{equation}
\label{defn_g}
\left ( \Phi \left ( \sqrt{\frac{F_{m+1}^{-1}(u)}{m}} \, d(x) \right )\,  - \Phi \left( \sqrt{\frac{F_{m+1}^{-1}(u)}{m}} \, t(m)  \right )  \right ) f_V \left ( x \sqrt{\frac{F_{m+1}^{-1}(u)}{m}}; \, ||\gamma || \right )
\end{equation}
for all $(x,u) \in [0,k] \times [0,1)$ and the limit
of \eqref{defn_g} as $u$ approaches 1 from
below for $u=1$ and all $x \in [0,k]$. Thus
$g(1;x,||\boldsymbol{\gamma}||) = 0$ for all $x \in [0,k]$.
%
%
%
%

\bigskip

\noindent \textbf{Appendix C. Theorem 2 and its proof}

\medskip

The following theorem provides a computationally-convenient expression for
the criterion $e(0; d)$.

\begin{theorem}

The criterion $e(0; d)$ is equal to
\begin{equation}
\label{criterion}
1 + \frac{2^{3/2} \, s^{s/2} \, \Gamma((s+m+1)/2)}{t(m) E(W) \, \Gamma(m/2) \, \Gamma(s/2)}
\, \int_0^k \big ( d(x) - t(m) \big ) \, x^{s-1} \, \frac{m^{m/2}}{(s x^2 + m)^{(s+m+1)/2}} \, dx
\end{equation}
where $\Gamma$ denotes the gamma function.

\end{theorem}

\medskip

\begin{proof}

\noindent The proof of this theorem uses Theorem 1 (b).
It follows from \eqref{sc_exp_len_final} that
\begin{equation*}
e (0; d ) - 1
= \frac{1}{t(m) E(W)} \int_0^{\infty} \int_0^k \big( d(x) - t(m) \big) \, f_V(wx; 0)
\, dx \, w^2 \, f_W(w) \, dw
\end{equation*}
Note that $f_V(v;0) = 2 s v f_s(s v^2)$,
where $f_s$ denotes $\chi_s^2$ probability density function.
Interchanging the order of integration, we obtain
\begin{equation*}
e (0; d ) - 1
= \frac{2s}{t(m) E(W)}\int_0^k  \big( d(x) - t(m) \big) \, x \,  \int_0^{\infty}   \, f_s \big((sx^2) w^2 \big)
\,  w^3 \, f_W(w) \, dw \, dx .
\end{equation*}

\begin{lemma}

For each $y>0$,
\begin{equation}
\label{lemma1_result}
\int_0^{\infty} f_s(yw^2) \, w^3 \, f_W(w) \, dw =
\frac{2^{1/2} \, m^{m/2} \, y^{(s/2)-1} \, \Gamma((s+m+1)/2)}{(y+m)^{(s+m+1)/2} \, \Gamma(m/2) \, \Gamma(s/2)}.
\end{equation}

\end{lemma}

\begin{proof}

Note that $f_W(w)=2mwf_m(mw^2)$,
where $f_m$ denotes the $\chi_m^2$ probability density function.
Substituting the expressions for $f_s$ and $f_W$, we obtain
\begin{equation*}
\int_0^{\infty} f_s(yw^2) \, w^3 \, f_W(w) \, dw =  \frac{m^{m/2}\, y^{(s/2)-1}} {2^{(s+m-2)/2} \,
 \Gamma(m/2)\, \Gamma(s/2)} \int_0^{\infty} e^{-(y+m)w^2/2} \, w^{s+m} \, dw.
\end{equation*}
By (A2.1.3) of Box and Tiao (1973), this is equal to the right-hand side of \eqref{lemma1_result}.

\end{proof}

\noindent It follows from this lemma that
\begin{equation*}
e (0; d ) - 1
= \frac{2^{3/2} \, s^{s/2} \, \Gamma((s+m+1)/2)}{t(m) E(W) \, \Gamma(m/2) \, \Gamma(s/2)} \int_0^k
\big( d(x) - t(m) \big) \, x^{s-1} \,  \frac{m^{m/2}}{(sx^2+m)^{(s+m+1)/2}} \,  dx.
\end{equation*}

\end{proof}

\bigskip

\noindent \textbf{Appendix D. Some simple results on confidence interval performance}

\medskip

In this appendix we consider the confidence interval
\begin{equation*}
J(d) = \Big [ \hat \Theta - \sqrt{v_{11}} \, \hat \Sigma \, d ( \sqrt{F} ), \,
\hat \Theta + \sqrt{v_{11}} \, \hat \Sigma \, d ( \sqrt{F} ) \Big ],
\end{equation*}
where $d: [0,\infty) \rightarrow (0,\infty)$. We make additional requirements of $d$, as needed.
We state some simple results about the performance of this confidence interval. The proofs of
these results are straightforward and are omitted, for the sake of brevity.

Theorems 3 and 4 concern the expected length of $J(d)$. Theorem 3 is used in the proof of Theorem 4.

\begin{theorem}

Suppose that $d_1(x) \ge d_2(x)$ for all $x \ge 0$. Then
\begin{equation*}
E \big(\text{length of } J(d_1) \big) \ge E \big(\text{length of } J(d_2) \big)
\ \ \text{for all }\ ||\gamma||.
\end{equation*}

\end{theorem}

\medskip

\begin{theorem}

Suppose that $d_1(x) \ge d_2(x)$ for all $x \ge 0$ and that there exists $\epsilon > 0$ and
an interval $[a,b]$ (where $0 \le a < b$) such that $d_1(x) > d_2(x) + \epsilon$ for all
$x \in [a,b]$. Then
\begin{equation*}
E \big(\text{length of } J(d_1) \big) > E \big(\text{length of } J(d_2) \big)
\ \ \text{for all }\ ||\gamma||.
\end{equation*}

\end{theorem}

\medskip

Theorems 5 and 6 concern the coverage probability of $J(d)$. Theorem 5 is used in the proof of Theorem 6.

\begin{theorem}

Suppose that $d_1(x) \ge d_2(x)$ for all $x \ge 0$. Then
\begin{equation*}
P \big(\theta \in J(d_1) \big)
 \ge P \big(\theta \in J(d_2) \big)
\ \ \text{for all }\ ||\gamma||.
\end{equation*}

\end{theorem}

\medskip

\begin{theorem}

Suppose that $d_1(x) \ge d_2(x)$ for all $x \ge 0$ and that there exists $\epsilon > 0$ and
an interval $[a,b]$ (where $0 \le a < b$) such that $d_1(x) > d_2(x) + \epsilon$ for all
$x \in [a,b]$. Then
\begin{equation*}
P \big(\theta \in J(d_1) \big)
 > P \big(\theta \in J(d_2) \big)
\ \ \text{for all }\ ||\gamma||.
\end{equation*}

\end{theorem}

\medskip

These theorems have the following three consequences.

\begin{corollary}

Suppose that $d$ is continuous. If $d(0) < t(m)$ and $d(x) \le t(m)$ for all $x > 0$ then
\begin{equation*}
P (\theta \in J(d) )
 < 1-\alpha
\ \ \text{for all }\ ||\gamma||.
\end{equation*}

\end{corollary}









\baselineskip=19pt

\bigskip

\noindent {\bf References}

\smallskip

\rf Bickel, P.J., 1984. Parametric robustness: small biases can be
worthwhile. Annals of Statistics 12, 864--879.



\rf Box, G.E.P., Tiao, G.C., 1973. Bayesian Inference in Statistical Analysis. Wiley, New York.

\rf Brown, L.D., Casella, G., Hwang, J.T.G., 1995. Optimal confidence sets, bioequivalence and the Limacon of Pascal.
Journal of the American Statistical Association 90, 880--889.


\rf Farchione, D., Kabaila, P., 2008. Confidence intervals for the normal mean utilizing prior information.
Statistics and Probability Letters 78, 1094--1100.

\rf Farchione, D., Kabaila, P., 2012. Confidence intervals in regression centred on the SCAD estimator.
Statistics and Probability Letters 82, 1953--1960.


\rf Goutis, C., Casella, G., 1991. Improved invariant confidence intervals for the normal variance.
Annals of Statistics 19, 2015--2031.

\rf Hinkelmann, K., Kempthorne, O., 1994. Design and Analysis of Experiments, revised edition.
John Wiley, New York.

\rf Hodges, J.L., Lehmann, E.L., 1952. The use of previous
experience in reaching statistical decisions. Annals of
Mathematical Statistics 23, 396--407.


\rf Kabaila P., 2009. The coverage properties of confidence regions after model
selection. International Statistical Review 77, 405--414.

\rf Kabaila P., 2011. Admissibility of the usual confidence interval for the normal
mean. Statistics and Probability Letters 81, 352--359.

\rf Kabaila, P., Farchione, D., 2012. The minimum coverage probability of confidence intervals in regression after a preliminary F test.
Journal of Statistical Planning and Inference 142, 956--964.

\rf Kabaila, P.,  Giri, K., 2009a. Confidence intervals in regression utilizing prior
information. Journal of Statistical Planning and Inference
139, 3419--3429.

\rf Kabaila, P., Giri, K., 2009b. Large-sample confidence intervals
for the treatment difference in a two-period crossover trial, utilizing
prior information.
Statistics and Probability Letters 79, 652--658.

\rf Kabaila, P., Giri, K., 2013. Simultaneous confidence interval for the population cell
means, for two-by-two factorial data, that utilize uncertain prior information.
To appear in Communications in Statistics - Theory and Methods.

\rf Kabaila, P.,  Giri, K., Leeb, H., 2010. Admissibility of the usual
confidence interval in linear regression. Electronic Journal of Statistics
4, 300--312.

\rf Kempthorne, P.J., 1983. Minimax-Bayes compromise estimators. In
1983 Business and Economic Statistics Proceedings of the
American Statistical Association, Washington DC, pp.568--573.

\rf Kempthorne, P.J., 1987.  Numerical specification of
discrete least favourable prior distributions.  SIAM
Journal on Scientific and Statistical Computing 8, 71--184.


\rf Kempthorne, P.J., 1988. Controlling risks under different loss
functions: the compromise decision problem. Annals of
Statistics 16, 1594--1608.






\rf Maatta, J.M., Casella, G., 1990. Decision-theoretic estimation. Statistical
Science 5, 90--120.

\rf Mead, R., 1988. The Design of Experiments.
Cambridge University Press, Cambridge.

\rf Pratt, J.W., 1961. Length of confidence intervals. Journal of
the American Statistical Association 56, 549--657.

\rf Puza, B., O'Neill, T., 2006a. Generalised Clopper-Pearson confidence intervals for the binomial proportion.
Journal of Statistical Computation and Simulation 76, 489--508.

\rf Puza, B., O'Neill, T., 2006b. Interval estimation via tail functions.
Canadian Journal of Statistics 34, 299--310.




\end{document}